\documentclass[11pt]{amsart}

\usepackage{amsmath} 
\usepackage{amsthm} 
\usepackage{amsfonts} 
\usepackage{amssymb} 
\usepackage{mathtools} 
\usepackage{thm-restate} 

\usepackage[style=alphabetic, maxnames=100, maxalphanames=100]{biblatex} 
\DeclareFieldFormat{postnote}{\mknormrange{#1}} 

\usepackage{bm} 
\usepackage{bbm} 

\usepackage{array} 
\usepackage[inline]{enumitem} 

\usepackage{microtype} 
\usepackage{lmodern} 

\usepackage[in]{fullpage} 

\usepackage[%
  plainpages=false,
  pdfpagelabels,
  colorlinks,
  citecolor=black,
  linkcolor=black,
  urlcolor=black,
  filecolor=black,
  bookmarksopen=false
]{hyperref} 
\usepackage[all]{hypcap} 
\hypersetup{hypertexnames=false}

\newtheoremstyle{thmstyle}
  {\medskipamount}
  {\smallskipamount}
  {\slshape}
  {0pt}
  {\bfseries}
  {.}
  { }
  {\thmname{#1}\thmnumber{ #2}{\normalfont\thmnote{ (#3)}}}
  
\newtheoremstyle{plainstyle}
  {\medskipamount}
  {\smallskipamount}
  {\rmfamily}
  {0pt}
  {\bfseries}
  {.}
  { }
  {\thmname{#1}\thmnumber{ #2}{\normalfont\thmnote{ (#3)}}}

\theoremstyle{thmstyle}
\newtheorem{theorem}{Theorem}[section]
\newtheorem{lemma}[theorem]{Lemma}
\newtheorem{definition}[theorem]{Definition}
\newtheorem{claim}{Claim}[theorem]
\newtheorem{remark}[theorem]{Remark}
\newtheorem{conclusion}[theorem]{Conclusion}

\theoremstyle{plainstyle}
\newtheorem{discussion}[theorem]{Discussion}
\newenvironment{proofof}[1]{\begin{proof}[Proof of #1.]}{\end{proof}}

\newlist{enumdef}{enumerate}{3}
\setlist[enumdef,1]{before={\leavevmode}, label={\arabic*.}, ref={\thetheorem.\arabic*}}
\setlist[enumdef,2]{before={\leavevmode}, label={\alph*.}, ref={\thetheorem.\theenumi.\alph*}}
\setlist[enumdef,3]{before={\leavevmode}, label={\roman*.}, ref={\thetheorem.\theenumi.\theenumii.\roman*}}

\setlist[enumerate,1]{label={\arabic*.}, ref={(\arabic*)}} 
\setlist[enumerate,2]{label={\alph*.}, ref={(\theenumi.\alph*)}} 

\numberwithin{equation}{section} 

\makeatletter
\newcommand{\combinesymbols}[3][\mathord]{#1{\mathpalette\combinesymbols@paletted{{#2}{#3}}}}
\newcommand{\combinesymbols@paletted}[2]{\combinesymbols@internal{#1}#2}
\newcommand{\combinesymbols@internal}[3]{\ooalign{\hss$\m@th #1#2$\hss\cr\hss$\m@th #1#3$\hss}}
\makeatother

\let\epsilon\varepsilon

\newcommand{\rn}{\bm}
\newcommand{\df}{\stackrel{\text{def}}{=}}
\newcommand{\symdiff}{\mathbin{\triangle}}
\newcommand{\otriangle}{\combinesymbols[\mathbin]{\bigcirc}{\triangle}}
\newcommand{\rest}{\mathord{\vert}}
\newcommand{\floor}[1]{\lfloor#1\rfloor}

\newcommand{\conc}{\mathbin{{}^\smallfrown}}
\newcommand{\fn}[2]{\mathord{{}^{#1}#2}}

\newcommand{\rstr}{\mathord{\upharpoonright}}

\DeclareMathOperator{\UC}{UC}
\DeclareMathOperator{\VC}{VC}
\DeclareMathOperator{\SVC}{SVC}
\DeclareMathOperator{\VCN}{VCN}
\DeclareMathOperator{\ev}{ev}
\DeclareMathOperator{\lgn}{lg}

\newcommand{\PP}{\mathbb{P}}
\newcommand{\RR}{\mathbb{R}}

\newcommand{\One}{\mathbbm{1}}

\newcommand{\cB}{\mathcal{B}}
\newcommand{\cC}{\mathcal{C}}
\newcommand{\cE}{\mathcal{E}}
\newcommand{\cH}{\mathcal{H}}
\newcommand{\cK}{\mathcal{K}}
\newcommand{\cL}{\mathcal{L}}
\newcommand{\cP}{\mathcal{P}}

\newenvironment{centerquote}[1][]{%
  \smallskip
  \begin{quotation}
    #1
}{%
  \end{quotation}
  \smallskip
}

\def\Cervonenkis{\v{C}ervonenkis}

\def\alt{D}
\let\emph\textit

\newcommand{\WARNING}[2][Warning]{
  \typeout{^^J#1 on line \the\inputlineno: #2^^J}
  \textbf{#1 on line \the\inputlineno: #2}
}

\title{%
  Some model-theoretic consequences\\
  of high-arity uniform convergence, part~I%
}
\author{%
  Leonardo N.~Coregliano
  \and
  Maryanthe Malliaris
}

\thanks{%
  LNC: Research partially supported by the Suzuki Postdoctoral Fellowship.
  \newline\indent
  MM: Research partially supported by NSF-BSF~2051825.
}

\address{Department of Mathematics, University of Chicago, 5734 S.~University Avenue, Chicago, IL 60637, USA}
\email{lenacore@uchicago.edu}
\address{Department of Mathematics, University of Chicago, 5734 S.~University Avenue, Chicago, IL 60637, USA}
\email{mem@math.uchicago.edu}

\begin{document}

\begin{abstract}
  We show that certain families of sets in $\RR^2$ (or $\RR^n$) which are neither definable nor have bounded $\VC$-dimension are
  nonetheless uniformly approximately definable in the real field, an o-minimal structure.
\end{abstract}

\maketitle

The aim of this paper is to show that certain families of sets in $\RR^2$ which are neither definable nor have bounded
$\VC$-dimension are nonetheless uniformly approximately definable in the real field, an o-minimal structure, in a very natural
way. We make use of a high-arity uniform convergence statement proved in the context of high-arity PAC learning~\cite{CM24},
although for most items in the present paper we just use a special case which could be extracted from work of Livni--Mansour on
graph-based discriminators~\cite{LM19a,LM19b} (as we explain in the appendix on uniform convergence).
 
We thank J.~Pila and L.~van den Dries for some thoughtful remarks on earlier versions and we note our appreciation
for the work of Karpinski--Macintyre~\cite{KM00}.

\section{Some motivation: Reconstruction from small random samples}

This section is not necessary for reading the proofs, although it tries to communicate the main idea. First some context. In
theoretical computer science, Valiant's theory of PAC learning~\cite{Val84} has to do with the problem of when possibly
infinitary objects (presented as subsets of a given set) can be reconstructed from a very small random sample. Learning happens
when the reconstruction is, with high probability, ``good enough'' (hence the name Probably Approximately Correct). The
connection with model theoretic approximation was already seen in the very interesting work of Karpinski and
Macintyre~\cite{KM00}.
 
For better or worse, however, \emph{not so many classes can be PAC learned}, because the fundamental theorem in the subject
(Vapnik-Chervonenkis~\cite{VC71,VC15}, Blumer--Ehrenfeucht--Haussler--Warmuth~\cite{BEHW89}, Natarajan~\cite{Nat89}; see
also~\cite[\S 6.4]{SB14} for a modern presentation) gives a characterization: a family of sets is PAC learnable if and only if
it has finite $\VC$-dimension.
 
Writing~\cite{CM24}, we considered learning higher-dimensional objects. This required understanding why one should reject the
conventional wisdom to simply code the higher-dimensional families as families of sets and apply the usual PAC theory. Indeed,
here is an interesting point, which became the thesis of~\cite{CM24}: many natural families of infinite (unary) $\VC$-dimension
have geometric or other mathematical structure \emph{in higher dimensions} which it should be possible to leverage in order to
learn (in the sense of achieving some good reconstruction from typical sample), despite the PAC theory classifying them as
``unlearnable.''

To achieve this higher reconstruction we had to rebuild the learning machine to allow for structured correlation in the sample
which could pick up more information. The main work of~\cite{CM24,CM25a} was thus to state and prove a higher version of a
fundamental theorem (with its many different equivalent conditions), for this new correlated sampling, showing that such higher
learning was possible and precisely characterizing the larger class for which it could be done.

One illustration of this correlated sampling from a model-theoretic point of view is the following: Suppose we are given a
family of models on the same domain. The adversary chooses a measure on the domain and M in the family, both unknown to us.
Based on receiving a small random sample of points in the domain along with the information of the submodel induced by $M$, we
try to guess the rest of $M$. In doing so note it is generally not sufficient to guess $M$ up to isomorphism (the family may
even consist of many automorphic images of the same model, or may include differences which become trivial under automorphism,
such as choosing a specific forking extension): this is an important point where the information available to learning differs
from the information available to property testing or regularity. A more precise illustration will be developed in the present
paper using as a key ingredient the uniform convergence statement in the Appendix which characterizes when grid-based
reconstruction will be accurate.

\section{Global conventions}
\label{sec:global-conventions}

Here are our global conventions in this paper: 
\begin{enumerate}
\item $\cL\df\{{+},{\times},{-}, 0, 1, {<}\}$ is the language of ordered rings.
\item Unless otherwise stated, all measures we use here are Borel probability measures on $\RR$, or product measures arising
  from these. In the case of countable (i.e., countably infinite) or finite sets, we always take the discrete $\sigma$-algebra.
\item We work in the $\cL$-model $M = (\RR; {+}, {\times}, {-}, 0, 1, {<})$, that is, in the reals as an ordered field. Hence we
  are in a standard Borel space.
\item If the reader prefers: what we actually use is that $M$ is an ordered field whose order topology is Polish and that each
  element of $\cC$ belongs to the associated Borel $\sigma$-algebra. However, we need the associated standard Borel space so our
  possibilities are in any case strongly constrained both in terms of cardinality and up to homeomorphism.
\item We will be reasoning about families $\cC$: for us $\cC$ denotes a family of subsets of $\RR^2$ such that each set
  $C\in\cC$ is measurable.
\item $\cH$ is a hypothesis class (in this paper, a family of subsets of $\RR^2$, with some measurability conditions, see
  Definition~\ref{def:hypothesis}).
\item $\cH\subseteq\cC$.  
\end{enumerate}
We will focus on $\RR^2$ since it is illustrative without being notationally complex; but we point out that all theorems
presented in this paper extend to $\RR^n$ in the obvious way.

\section{Meta-theorem}
\label{sec:meta-theorem} 

A main ingredient in the proofs will be high-arity uniform convergence; a simplified form of it is stated as
Theorem~\ref{thm:twoUCinformal} below, while a formal one (Theorem~\ref{thm:twoUC}) is stated and proved in
Appendix~\ref{sec:UC}.

Note that in this section, we are proving a theorem based on certain sufficient conditions. In our applications we will have to
check that these conditions are met.

\begin{definition}\label{def:svc}
  The family $\cH$ has \emph{finite slicewise $\VC$-dimension}, or \emph{finite $\SVC$} for short, when there is some $d <
  \omega$ so that any slice of the family along any axis-parallel line produces a family of $\VC$-dimension $\leq d$ (see
  Definition~\ref{def:VC} for the definition of $\VC$-dimension).

  More formally, for every $w,z_0,\ldots,z_d\in\RR$, there exist $A,B\subseteq d+1$ such that for every $H\in\cH$, we have
  \begin{align*}
    \{i < d+1 : (w,z_i)\in H\} & \neq A, &
    \{i < d+1 : (z_i,w)\in H\} & \neq B.
  \end{align*}
\end{definition}

Since our families are not assumed to be definable, we note that $\SVC$ (see Definition~\ref{def:svc}) is really a property of
the family and not of the individual instances. For instance, if $\cH$ contains only one set then it will have finite $\SVC$
regardless of the complexity of the set.

In the next definition, to say that $\cH$ is a hypothesis class means that it is a family of subsets of $\RR^2$ [both 
``$\RR$'' and ``2'' are assumptions specific to our present paper] which satisfies
some additional measurability assumptions [with these we are recalling usual assumptions 
needed for high-arity uniform convergence, and in general for learning algorithms to make sense]. 

\begin{definition}\label{def:hypothesis}
  Suppose $\cH\subseteq\cP(\RR^2)$ is a set of subsets of $\RR^2$. We say that $\cH$ is a hypothesis class when it is equipped
  with a $\sigma$-algebra such that:
  \begin{enumerate}
  \item the evaluation map $\ev\colon\cH\times\RR^2\to\{0,1\}$ given by
    \begin{gather*}
      \ev(H,\overline{x})\df \One[\overline{x}\in H] \df
      \begin{dcases*}
        1, & if $\overline{x}\in H$,\\
        0, & if $\overline{x}\notin H$,\\
      \end{dcases*}
    \end{gather*}
    is a measurable function, where $\cH\times\RR^2$ is equipped with the product $\sigma$-algebra of $\cH$ and the Borel
    $\sigma$-algebra of $\RR^2$ (in particular, this implies that every $H\in\cH$ has to be Borel-measurable).
  \item for every $H\in\cH$, the set $\{H\}$ is measurable.
  \item for every standard Borel space $\Upsilon$ and every measurable set $Y\subseteq\cH\times\Upsilon$, the projection of $Y$
    onto $\Upsilon$, i.e. the set $\{\upsilon\in\Upsilon : \exists H\in\cH, (H,\upsilon)\in Y\}$, is measurable in every
    completion of a Borel probability measure on $\Upsilon$ (this is called ``universally measurable'')\footnote{The easiest way
    of ensuring that this property holds (and indeed is how we will show this property holds in our arguments) is by equipping
    $\cH$ with a topology that makes it a Suslin space and then taking the Borel $\sigma$-algebra.}.
  \end{enumerate}
\end{definition}

\begin{remark}\label{rmk:adding-F}
  If $\cH\subseteq\cP(\RR^2)$ is a hypothesis class and $F\subseteq\RR^2$ is a measurable set then $\cH\cup\{F\}$ remains a
  hypothesis class (when equipped with the coproduct $\sigma$-algebra).
\end{remark}

At this point, we can give an informal statement of the key ingredient, Theorem~\ref{thm:twoUC} (a formal statement and a proof
from the results of~\cite{CM24} can be found in Appendix~\ref{sec:UC}):
\begin{theorem}[Informal version of Theorem~\ref{thm:twoUC}]\label{thm:twoUCinformal}
  Given $\epsilon,\delta > 0$ and $d < \omega$, there exists $m^{\UC} = m^{\UC}(\epsilon,\delta,d)$ such that for all Borel
  probability measures $\mu_0$ and $\mu_1$ over $\RR$, every hypothesis class $\cH\subseteq\cP(\RR^2)$ with $\SVC(\cH)\leq d$,
  (optional) every choice of one additional measurable $F\subseteq\RR^2$, and every positive integer $m\geq m^{\UC}$, if we
  sample $\rn{\overline{a}}\sim\mu_0^m$ and $\rn{\overline{b}}\sim\mu_1^m$ independently, then with probability at least
  $1-\delta$, the resulting grid $\{ (\rn{a}_i, \rn{b}_j) : i,j < m \}$ will be \emph{two-way $\epsilon$-representative} for
  $\cH\cup\{F\}$ in the sense that for every $H_0,H_1\in\cH\cup\{F\}$, the quantities
  \begin{align*}
    L_{\mu_0,\mu_1}(H_0,H_1)
    & \df
    (\mu_0\otimes\mu_1)(H_0\symdiff H_1),
    &
    L_{\rn{\overline{a}},\rn{\overline{b}}}(H_0,H_1)
    & \df
    \frac{\lvert\{(i,j)\in m\times m : (\rn{a}_i,\rn{b}_j)\in H_0\symdiff H_1\}\rvert}{m^2}
  \end{align*}
  are $\epsilon$-close.
\end{theorem}
Informally, the above says that in a typical sufficiently large sample, the relative size
$L_{\rn{\overline{a}},\rn{\overline{b}}}(H_0,H_1)$ of the symmetric difference $H_0\symdiff H_1$ between two sets
$H_0,H_1\in\cH\cup\{F\}$ on the grid arising from the sample is $\epsilon$-close 
to the measure $L_{\mu_0,\mu_1}(H_0,H_1)$ of the of $H_0\symdiff H_1$.

The quantities $L_{\mu_0,\mu_1}(H_0,H_1)$ and $L_{\rn{\overline{a}},\rn{\overline{b}}}(H_0,H_1)$ are called the \emph{total
loss} (or \emph{total distance}) and the \emph{empirical loss} (or \emph{empirical distance}) on the sample
$(\rn{\overline{a}},\rn{\overline{b}})$, respectively.

\medskip

Let us now fix some notation for dealing with labeled grids. Given sequences $\overline{a} = \langle a_i : i < m \rangle$,
$\overline{b} = \langle b_j : j < m \rangle$ from $\RR$, the corresponding grid is $\Gamma = \Gamma(\overline{a}, \overline{b})
= \{ (a_i, b_j) : i, j < m \}$ in $\RR^2$. (Since our sequences may contain repetitions, such grids may be rectangular rather
than square, and of course the rows and columns need not be evenly spaced.) A labeling is a map from $\Gamma$ to $\{0,1\}$. The
definability we will find is based on the information received from a labeled grid. It will be convenient to partition the
parameter variables of our formulas into grid variables and label variables, so our formulas will typically be called
$\varphi_m(x,y; \overline{z}, \overline{w})$ where $\overline{z}$ has length $\lgn(\overline{z})=2m$ (note $m$ is from the
subscript in $\varphi_m$) and $\overline{w}$ has length $\lgn(\overline{w})=m^2$, and the intention is that
$\overline{a}\conc\overline{b}$ will fill $\overline{z}$ and the labels\footnote{Since $0$, $1$ are constants in $\cL$ there is
no loss of generality in using these as our labels, but we could have used $2m^2$ arbitrary elements and coded yes or no using
equalities and inequalities.} will fill $\overline{w}$. One could add other parameters if needed.

\begin{definition}
  Given a grid $\Gamma=\Gamma(\overline{a},\overline{b})\subseteq\RR^2$ and a set $C\in\cC$, the \emph{labeling of $\Gamma$
  induced by $C$} is the map $\pi\colon\Gamma\to\{0,1\}$ given by $\pi(a_i,b_j) = 1$ if and only if $(a_i,b_j)\in C$.
\end{definition}

\begin{definition}\label{def:Hdr}
  Given $\cC$ and a hypothesis class $\cH$, say that
  \begin{centerquote}[\centering\it]
    $\cC$ is $\cH$-definably realizable on grids
  \end{centerquote}
  when there exists a family of $\cL$-formulas $\{\varphi_m(x,y; \overline{z}_m, \overline{w}_m) : m < \omega \}$ such that: 
  \begin{enumdef} 
  \item for each $m < \omega$, $\varphi_m = \varphi_m(x, y; \overline{z}_m, \overline{w}_m)$, $\lgn(\overline{z}_m) = 2m$,
    $\lgn(\overline{w}_m) = m^2$, and the intent is:
    \begin{enumdef}
    \item for any choice of parameters, this formula will define a subset of $\RR^2$;
    \item the variables $\overline{z}_m$ input $2m$ elements of $\RR$, from which we build a grid in $\RR^2$ of size $m^2$
    \item the variables $\overline{w}_m$ input a label of $0$ or $1$ for each element of the grid.
    \end{enumdef}
    We may drop the subscripts $m$ from the parameter variables for easier reading. 
  \item\label{def:Hdr:varphiincH} for each $m < \omega$, for all parameters arising from labelings, $\varphi_m$ takes values in
    $\cH$ meaning:
    \begin{centerquote}
      for every choice of $\overline{a} = \langle a_i : i < m \rangle$, $\overline{b} = \langle b_j : j < m\rangle$, and
      $C\in\cC$, which then determines a sequence $\overline{\ell}\in\fn{m^2}{\{0,1\}}$ which records the labeling of
      $\Gamma(\overline{a},\overline{b})$ by $C$, we have that the resulting formula
      $\varphi_m(x,y;\overline{a}\conc\overline{b}, \overline{\ell})$ defines a set in $\RR^2$ which belongs\footnote{This item
      does not a priori require any connection to the $C$ we chose, just that the output is always in our $\cH$.} to $\cH$.
    \end{centerquote}
  \item\label{def:Hdr:count} the family $\{\varphi_m : m < \omega\}$ witnesses that $\cC$ is definably realizable on any grid in
    the natural way in the sense that:
    \begin{centerquote}
      for every choice of $\overline{a} = \langle a_i : i < m\rangle$, $\overline{b} = \langle b_j : j < m\rangle$, $C\in\cC$,
      letting $\overline{\ell}$ record the labeling of the grid $\Gamma = \Gamma(\overline{a}, \overline{b})$ by $C$, we have
      that the symmetric difference of $\varphi_m(x,y; \overline{a}\conc\overline{b}, \overline{\ell})$ and $C$ on $\Gamma$ has
      counting measure zero.
    \end{centerquote}
  \end{enumdef}
\end{definition}

\begin{remark}
  We point out that in Definition~\ref{def:Hdr} there is one $\varphi_m$ for each $m$; the formulas do not depend on the choice
  of $C\in\cC$.
  
  Also note the condition of Definition~\ref{def:Hdr} will be a sufficient condition. One way to
  achieve~\ref{def:Hdr:varphiincH} in practice will be to first select or construct an appropriate family of formulas $\varphi_m$ in
  the language of ordered rings and then take our hypothesis class $\cH$ to be all sets defined by instances of the $\varphi_m$,
  which we would then just have to verify satisfies the measurability conditions for being a hypothesis class (sometimes we will
  make $\cH$ slightly larger to make measurability arguments easier).
\end{remark}

We arrive to our meta-theorem \ref{thm:meta-thm}, which says: given such objects, and assuming that $\cC$ (or just $\cH$) has
finite slicewise $\VC$-dimension, the following holds. For any $\epsilon > 0$, we can choose $\varphi_m$ (here $m$ and hence
$\varphi_m$ depends only on $\epsilon$, and not on the measures nor on the choice of $C \in \cC$) so that for any choice of
Borel probability measures $\mu_1,\mu_2$ on $\RR$ and any choice of $C\in\cC$, some instance of $\varphi_m$ with parameters in a
sampled $m \times m$ grid defines a set which is $\epsilon$-close to $C$ on $\RR^2$ in the sense of the product measure
$\mu_1\otimes\mu_2$. In fact, we can even find one grid which works for all $C$: there is a choice of
$\overline{a},\overline{b}$ which depends on the measures but not on $C$ so that once we are given $C$, we simply complete the
parameters by adding the labeling of $\Gamma(\overline{a},\overline{b})$ induced by this $C$. And such
$(\overline{a},\overline{b})$ is easy to find.

Observe that several things are happening in this theorem: the elements of the family $\cC$ are acquiring an approximate
definition, and that definition uses parameters in a \emph{sampled grid} of bounded size which moreover, as just noted,
satisfies a strong two-way uniformity property.
 
The statement repeats the hypotheses for easier quotation. 

\begin{theorem}\label{thm:meta-thm}
  For every $\epsilon > 0$ and $d < \omega$, there exists a positive integer $m < \omega$ such that the following hold:

  Let $\cH\subseteq\cC$ be families of subsets of $\RR^2$, all of whose elements are measurable, and which satisfy:
  \begin{enumerate}[label={\alph*.}, ref={(\alph*)}]
  \item $\cH$ is a hypothesis class with slicewise $\VC$-dimension at most $d$;
  \item $\cC$ is $\cH$-definably realizable on grids\footnote{The proofs only explicitly use that $\cH$ has finite slicewise
  $\VC$-dimension, but this implies the same property for $\cC$ because of $\cH$-definable realizability. If $\cH$ has finite $\SVC$
  and $\cC$ does not, there will exist a grid that on one side repeats a vertex whose slice shatters a set larger than the
  slicewise $\VC$-dimension of $\cH$ and on the other side contains the set it shatters, and it will be impossible to get the
  counting measure of the symmetric difference to be zero.}.
  \end{enumerate}
  Then for all Borel probability measures $\mu_0$ and $\mu_1$ on $\RR$, the following hold:
  \begin{enumerate}
  \item\label{thm:meta-thm:weak} For every $C\in\cC$, there exist parameters $\overline{c}$ for $\varphi_m$ so that
    \begin{gather*}
      (\mu_0\otimes\mu_1)\bigl(\varphi_m(\RR^2; \overline{c})\symdiff C\bigr)\leq\frac{\epsilon}{2}.
    \end{gather*}
  \item\label{thm:meta-thm:strong} There exist $\overline{a}$, $\overline{b}\in\fn{m}{\RR}$, so that for \emph{every} $C\in\cC$,
    there exists $\overline{\ell}\in\fn{m^2}{\{0,1\}}$ such that
    \begin{gather*}
      (\mu_0\otimes\mu_1)
      \bigl(
      \varphi_m(\RR^2; \overline{a}\conc\overline{b}, \overline{\ell})\symdiff
      \bigr)
      \leq
      \epsilon.
    \end{gather*}
  \item\label{thm:meta-thm:whp} In fact, such grids are easy to find: if $\rn{\overline{a}}$ is sampled i.i.d.\ from $\mu_0$ and
    $\rn{\overline{b}}$ is sampled i.i.d.\ from $\mu_1$, then with probability at least $1/2$, $\rn{\overline{a}}$,
    $\rn{\overline{b}}$ will satisfy item~\ref{thm:meta-thm:strong}; and we could make $\delta$ arbitrarily close to $1$ by
    adding a dependence on $\delta < 1$ in the choice of $m$.
  \end{enumerate} 
\end{theorem}

Before we prove Theorem~\ref{thm:meta-thm}, let us stress the surprisingly strong order of quantification of
items~\ref{thm:meta-thm:strong} and~\ref{thm:meta-thm:whp}: the parameter $m$ that governs the length/complexity of the formula
depends only on $d$ and $\epsilon$, then the parameters $\overline{a}$ and $\overline{b}$ that determine the grid (but not its
labeling) depend only on ($d$, $\epsilon$,) $\cH$ and $\cC$ (but not on the specific set $C\in\cC$ that we are going to
approximately define) and finally, only after we are given $C$, we determine the labeling $\overline{\ell}$.

\begin{proof}
  At the very beginning of the proof, we are given $d < \omega$ and $\epsilon > 0$ and potentially a $\delta > 0$ for
  item~\ref{thm:meta-thm:whp} (the item is trivial when $\delta\geq 1$, so we assume $\delta < 1$ and for the other items, one
  can simply take any value of $\delta\in(0,1)$). We can then let $m^{\UC}$ be given by Theorem~\ref{thm:twoUCinformal} with
  $(d,\epsilon/2,\delta)$ and fix any positive integer $m\geq m^{\UC}$. We then are given probability measures $\mu_0$ and
  $\mu_1$ on $\RR$.

  \begin{description}[wide, itemsep={3ex}]
  \item[Item~\ref{thm:meta-thm:weak}] Suppose we are given $C\in\cC$. Applying Theorem~\ref{thm:twoUCinformal} with
    $(d,\epsilon/2,\delta)$ and taking $F=C$, we know that there exists some $\overline{u},\overline{v}\in\fn{m}{\RR}$ with
    $(\overline{u},\overline{v})$ being $\epsilon/2$-representative for $\cH\cup\{C\}$ with respect to $(\mu_0,\mu_1)$, that is,
    we have
    \begin{gather}\label{eq:meta-thm:weakrepr}
      \forall H_0,H_1\in\cH\cup\{C\},
      \left\lvert
      (\mu_0\otimes\mu_1)(H_0\symdiff H_1)
      -
      \frac{\lvert\{(i,j)\in m\times m : (u_i,v_j)\in H_0\symdiff H_1\}\rvert}{m^2}
      \right\rvert
      \leq
      \frac{\epsilon}{2}.
    \end{gather}

    Invoking definable realizability on grids, Definition~\ref{def:Hdr:count}, we know that for the labeling $\overline{\ell}$
    induced by $C$ on the grid $\Gamma=\Gamma(\overline{u},\overline{v})$, the formula
    $\varphi(x,y;\overline{u}\conc\overline{v},\overline{\ell})$ defines an element $H_C\in\cH$ that has empirical distance zero
    from $C$ on $\Gamma$ (i.e., $\lvert\{(i,j)\in m\times m : (u_i,v_j)\in H_C\symdiff C\}\rvert = 0$).

    Instantiating $\epsilon/2$-representativeness~\eqref{eq:meta-thm:weakrepr} of $(\overline{u},\overline{v})$ with $H_0\df
    H_C$ and $H_1\df C$, we conclude that $(\mu_0\otimes\mu_1)(H_C\symdiff H)\leq\epsilon/2$. This proves
    item~\ref{thm:meta-thm:weak} for $\overline{c} = \overline{u}\conc\overline{v}\conc\overline{\ell}$, and
    pedantically let us summarize:
    \begin{centerquote} 
      if we are given $\mu_0,\mu_1$ and $C\in\cC$, there is $\overline{c}$ somewhere in the model, more precisely coming from a
      small sampled grid which a priori depends on $C$, so that the definable set given by $\varphi_m(x,y;\overline{c})$ is an
      element of $\cH$ which is $\epsilon/2$-close to $C$ on $\RR^2$ in the sense of the product measure $\mu_0\otimes\mu_1$.
    \end{centerquote}
  \item[Item~\ref{thm:meta-thm:strong}] Here we improve the order of quantification: namely, our grid will no longer depend on
    $C\in\cC$. For this, apply Theorem~\ref{thm:twoUCinformal} with $(d,\epsilon/2,\delta)$ without any additional $F$ (or just
    take $F$ to be any element of $\cH$) to get an $\epsilon/2$-representative $(\overline{a},\overline{b})$ for $\cH$ with
    respect to $(\mu_0,\mu_1)$, that is:
    \begin{gather}\label{eq:meta-thm:strongrepr}
      \forall H_0,H_1\in\cH,
      \left\lvert
      (\mu_0\otimes\mu_1)(H_0\symdiff H_1)
      -
      \frac{\lvert\{(i,j)\in m\times m : (a_i,b_j)\in H_0\symdiff H_1\}\rvert}{m^2}
      \right\rvert
      \leq
      \frac{\epsilon}{2}.
    \end{gather}
    Let us call $\Gamma(\overline{a}, \overline{b})$ the \emph{fixed grid}.

    Now the adversary chooses $C\in\cC$. By item~\ref{thm:meta-thm:weak}, there is a choice of parameters
    $\overline{c}\in\fn{2m+m^2}{\RR}$ (arising from an a priori different grid) so that $\varphi_m(x,y;\overline{c})$ defines in
    $\RR^2$ a set which is a member of $\cH$, call it $H_0$, and which is $\epsilon/2$-close to $C$ on $\RR^2$ in the sense of
    the product measure $\mu_0\otimes\mu_1$.
 
    Label our fixed grid $\Gamma=\Gamma(\overline{a},\overline{b})$ with $H_0$ and let the sequence of labels be
    $\overline{\ell}$. Because $\cC$ and hence $\cH$ [recall $\cH \subseteq \cC$] is definably realizable on grids, the set
    defined by $\varphi_m(x,y;\overline{a}\conc\overline{b},\overline{\ell})$ belongs to $\cH$: call it $H_1$. By construction
    the symmetric difference of $H_0$ and $H_1$ on $\Gamma$ has counting measure zero. Now since the fixed grid $\Gamma$ is
    $\epsilon/2$-representative for elements of $\cH$ in the sense of the inset equation ~\eqref{eq:meta-thm:strongrepr}, we
    have that $(\mu_0\otimes\mu_1)(H_0\symdiff H_1)\leq\epsilon/2$ on $\RR^2$.

    Finally, applying the triangle inequality concludes the proof of item~\ref{thm:meta-thm:strong}:
    \begin{gather*}
      (\mu_0\otimes\mu_1)(C\symdiff H_1)
      \leq
      (\mu_0\otimes\mu_1)(C\symdiff H_0)
      +
      (\mu_0\otimes\mu_1)(H_0\symdiff H_1)
      \leq
      \frac{\epsilon}{2}
      +
      \frac{\epsilon}{2}
      =
      \epsilon.
    \end{gather*}
  \item[Item~\ref{thm:meta-thm:whp}] We simply need $(\overline{a},\overline{b})$ to be $\epsilon/2$-representative for $\cH$
    with respect to $(\mu_0,\mu_1)$ in the sense of Theorem~\ref{thm:twoUCinformal}, so this follows from the statement of the
    theorem.\footnote{The full statement with quantification incorporating $\delta$ is: for every $\delta,\epsilon > 0$ and
    every $d < \omega$, there exists $m < \omega$ such that for $\cH$ and $\cC$ as above and Borel probability measures $\mu_0$
    and $\mu_1$ on $\RR$, sampling $\rn{\overline{a}}$ i.i.d.\ from $\mu_0$ and $\rn{\overline{b}}$ i.i.d.\ from $\mu_1$
    (independently from $\rn{\overline{a}}$) satisfies item~\ref{thm:meta-thm:strong} with probability at least $1-\delta$.}
  \end{description}
\end{proof}

\begin{discussion}\label{disc:gridlabeling}
  Note that there is a difference regarding how the labelings of the corresponding grids are obtained in
  items~\ref{thm:meta-thm:weak} and~\ref{thm:meta-thm:strong} of Theorem~\ref{thm:meta-thm}. Namely, in the former item, the
  labeling is obtained directly from the element $C\in\cC$ that we are trying to approximate; but on the latter item, the
  labeling is obtained indirectly: it is the labeling of an element $H_0\in\cH$ that approximates $C\in\cC$ in a grid that we a
  priori did not see.
\end{discussion}

\section{Application to convex sets}

\begin{theorem}\label{thm:convex}
  Let $M = (\RR; {+}, {\times}, {-}, 0, 1, {<})$ be the real field and $\cC$ the family of Borel convex sets in the real plane.
  Then ``all $C\in\cC$ are approximately definable in $M$'' in the sense that: 
  \begin{centerquote}
    given $\epsilon > 0$ there is an $\cL$-formula $\varphi = \varphi(\epsilon) = \varphi(x,y; \overline{z}, \overline{w})$
    without hidden parameters [and hence a uniform finite bound on $\lgn(\overline{z}\conc\overline{w})$ which depends only on
      $\epsilon$] such that for any probability measures $\mu_0,\mu_1$ on $\RR$, for any $C\in\cC$, there are parameters
    $\overline{a}$ from $M$ with $\lgn(\overline{a}) = \lgn(\overline{z}) + \lgn(\overline{w})$ such that
    $(\mu_0\otimes\mu_1)(\varphi(\RR^2; \overline{a})\symdiff C) < \epsilon$.
  \end{centerquote}
  Indeed, the proof will show: there is an integer $m$ such that $2m = \lgn(\overline{z})$, $m^2 = \lgn(\overline{w})$ and given
  $\mu_0, \mu_1$ there is an $m \times m$ grid $\Gamma$ in $\RR^2$ so that for any choice of $C \in \cC$, the defining
  parameters $\overline{a}$ can always be taken from $\Gamma \cup \{ 0, 1 \}$.
\end{theorem}

Before working out the proof let us examine the statement: 

\begin{discussion} 
  The family of convex sets is not definable and does not even have finite $\VC$-dimension,\footnote{Given any $d$ points around
  the perimeter of a circle, the convex hull of any subset of these points will not contain the others: this gives arbitrary
  shattering by the convex sets.} so obviously we cannot expect to expand $\RR$ by this family and remain o-minimal. Yet the
  theorem says we can \emph{uniformly} approximate all elements in this family up to $\epsilon$ with instances of a single
  formula $\varphi$ in the language of ordered rings which only depends on $\epsilon$, not on the measures. Moreover we can do
  so while carefully controlling the provenance of the parameters needed for $\varphi$.
\end{discussion}

As a reality check, note that the following can be proved with some care using purely analytic arguments:
\begin{claim}\label{clm:rc}
  Given any convex set $C$, any probability measures $\mu_0,\mu_1$ on $\RR$ and $\epsilon > 0$, there exists a finitely
  generated convex set $F$ [i.e., $F$ is the convex hull of finitely many points] so that $(\mu_0\otimes\mu_1)(F\symdiff C) <
  \epsilon$.
\end{claim}

However, in Claim~\ref{clm:rc}, the specific number of points needed to generate the resulting $F$ a priori depends on the
measures, on $\epsilon$, and on $C$; it is not obvious how to remove the dependence on the measures. In other words, even for a
single $C$ it is a priori not obvious how to uniformly bound the number of parameters in a measure-independent way.

\begin{proofof}{Theorem~\ref{thm:convex}}
  We check the conditions of our meta-theorem, Theorem~\ref{thm:meta-thm}. We need to define $\cH$ and to verify that:
  \begin{enumerate*}[label={(\arabic*)}]
  \item\label{thm:convex:subset} $\cH\subseteq\cC$,
  \item\label{thm:convex:measurable} all elements of $\cC$ (and hence of $\cH$) are measurable,
  \item\label{thm:convex:SVC} $\cH$ has finite slicewise $\VC$-dimension,
  \item\label{thm:convex:defreal} $\cC$ is $\cH$-definably realizable on grids, and
  \item\label{thm:convex:class} $\cH$ is a hypothesis class.
  \end{enumerate*}

  \begin{description}[wide, itemsep={3ex}]
  \item[Item~\ref{thm:convex:subset}] We take $\cH$ to be the set of all compact convex sets, it is clearly contained in the
    set of all Borel convex sets $\cC$.
  \item[Item~\ref{thm:convex:measurable}] By assumption, all sets in $\cC$ are Borel, so they are measurable.
  \item[Item~\ref{thm:convex:SVC}] A fortiori, we have $\SVC(\cH)=\SVC(\cC)=2$. This is because a slice of a convex set is
    always an interval and the collection of all intervals of $\RR$ has $\VC$-dimension $2$.
  \item[Item~\ref{thm:convex:defreal}] For each $m < \omega$, let $\varphi_m(x,y; \overline{z}_m, \overline{w}_m)$ be a formula
    with the following properties. It takes in $2m$ points via $\overline{z}_m$ and from these generates a grid of size $m^2$.
    It takes in labels for the points of this grid via $\overline{w}_m$. It defines the convex hull of the points labeled $1$.
    Clearly there is such an $\cL$-formula.

    We now argue that $\cC$ is $\cH$-definably realizable. The fact that the property of Definition~\ref{def:Hdr:varphiincH}
    holds follows since $\varphi_m$ defines a finitely generated convex set, so it must be compact, hence an element of $\cH$.
    The property of Definition~\ref{def:Hdr:count} follows since when we use a convex set $C\in\cC$ to label a grid then the
    convex hull of the points labeled $1$ yield the same labels on the grid as $C$.
  \item[Item~\ref{thm:convex:class}] We now justify that $\cH$ is a hypothesis class. By Remark~\ref{rmk:adding-F}, it suffices
    to justify that the set $\cH'\df\cH\setminus\{\varnothing\}$ of non-empty compact sets is a hypothesis class.

    For this, we equip $\cH'$ with the Hausdorff metric $d_H$ (induced by Euclidean distance) and we claim that this induces a
    Polish topology on $\cH'$. Since the space $\cK$ of non-empty compact (but not necessarily convex) sets equipped with
    (topology induced by) Hausdorff metric is Polish (see~\cite[Theorem~4.25]{Kec95}), this follows since $\cH'$ is a closed
    subspace of $\cK$, which we prove in Lemma~\ref{lem:convexHausdorff} in Appendix~\ref{sec:Hausdorff}.

    Since $\cH'$ is a Polish space, it follows that every singleton $\{H\}$ is measurable and for every standard Borel space
    $\Upsilon$ and every measurable set $Y\subseteq\cH'\times\Upsilon$, the projection of $Y$ onto $\Upsilon$, i.e. the set
    $\{\upsilon\in\Upsilon : \exists H\in\cH', (H,\upsilon)\in Y\}$, is then a Suslin set, which is universally measurable.

    It remains to prove that the evaluation map $\ev\colon\cH'\times\RR^2\to\{0,1\}$ is measurable. Note that
    \begin{gather*}
      \ev(H,\overline{x})
      =
      \One[\overline{x}\in H]
      =
      \One[d(\overline{x},H) = 0],
    \end{gather*}
    where $d$ is the Euclidean distance in $\RR^2$ (the second equality follows since $H$ is compact, hence closed). Now,
    $d(\overline{x},H)$ is Lipschitz-continuous with respect to the Euclidean distance $d$ on $\overline{x}$ and the Hausdorff
    distance on $H$, which we prove in Lemma~\ref{lem:Lipschitz} in Appendix~\ref{sec:Hausdorff}. This implies that $\ev$ is
    measurable.
  \end{description}

  Thus, Theorem~\ref{thm:meta-thm} applies. In fact, as stated, Theorem~\ref{thm:convex} follows from
  Theorem~\ref{thm:meta-thm}\ref{thm:meta-thm:weak}. We could have made a stronger statement: for every choice of $\mu_0,\mu_1$
  there is a single $m \times m$ grid $\Gamma$ in $\RR^2$ so that \emph{every} Borel convex set $C\in\cC$ is $\epsilon$-close
  (in the sense of $\mu_0\otimes\mu_1$) to some definable set given by an instance of $\varphi_m$ with parameters from $\Gamma$.
\end{proofof}

\section{Application to families of bounded alternation}

The following implies finite slicewise $\VC$-dimension.\footnote{It is strictly stronger: consider a family consisting of a
single set which is the disjoint union of countably many circles with centers along a vertical line.} In writing this paper, we
have asked ourselves whether it should be understood as the correct condition for an ``approximate'' o-minimality including many
not definable families hovering over the real field.

\begin{definition}
  Let $\cC$ be a family of subsets of $\RR^2$. Call $\cC$ a family of bounded alternation if there is some $\alt < \omega$ so
  that no element of $\cC$ has more than $\alt$ alternations along any axis-parallel line. We may say $\cC$ is a family of
  $\alt$-bounded alternation to emphasize the value of $\alt$. Clearly $\cC$ has slicewise $\VC$-dimension at most $\alt+1$.

  \emph{Note} that any such $\cC$ is assumed to satisfy \S\ref{sec:global-conventions}.
\end{definition}

\begin{theorem}\label{thm:alternation}
  Let $M = (\RR; {+}, {\times}, {-}, 0, 1, {<})$ be the real field and let $\alt < \omega$. Let $\cC$ be a family of
  $\alt$-bounded alternation, each of whose elements is Borel measurable. Then ``all $C\in\cC$ are approximately definable in
  $M$'' in the sense that:

  \begin{centerquote}
    given $\epsilon > 0$ there is an $\cL$-formula $\varphi = \varphi(\alt,\epsilon) = \varphi(x,y; \overline{z}, \overline{w})$
    without hidden parameters [and hence a uniform finite bound on $\lgn(\overline{z}\conc\overline{w})$ which depends only on
      $\alt$ and $\epsilon$] such that for any probability measures $\mu_0,\mu_1$ on $\RR$, for any $C\in\cC$, there are
    parameters $\overline{a}$ from $M$ with $\lgn(\overline{a}) = \lgn(\overline{z}) + \lgn(\overline{w})$ such that
    $(\mu_0\otimes\mu_1)(\varphi(x,y; \overline{a})\symdiff C) < \epsilon$.
  \end{centerquote}
  Indeed, the proof will show: there is an integer $m$ such that $2m = \lgn(\overline{z})$, $m^2 = \lgn(\overline{w})$ and given
  $\mu_0, \mu_1$ there is an $m \times m$ grid $\Gamma$ in $\RR^2$ so that for any choice of $C \in \cC$, the defining
  parameters $\overline{a}$ can always be taken from $\Gamma \cup \{ 0, 1 \}$.
\end{theorem}

Before working out the proof, let us examine the statement and motivate the approach:

\begin{discussion} 
  Again in this theorem several things are happening: the family $\cC$ is acquiring a schema of approximate definition, which is
  uniform given $\epsilon$ and whose parameters will have a known provenance, and the defining formulas now need to be
  constructed in the language of ordered rings without any a priori knowledge of the family other than what is given by the
  bound on alternation and hence on the nature of the uniform convergence on grids. Since convex sets in the plane have
  $2$-bounded alternation, the statement of Theorem~\ref{thm:alternation} implies the statement of Theorem~\ref{thm:convex}.
  However the proofs are different. Here the defining formulas will be bounded unions of ``cells'' (in the present proof, our
  cells are half-open rectangles, see below), whereas before we had formulas tailor-made for the class of convex sets (defining
  the convex hull of a bounded number of points).
\end{discussion}

The right choice of $\cH$ is a key aspect of this proof. In order to approximate the elements of our family on a given finite
grid without introducing arbitrary alternation, we appeal to an auxiliary (offset) grid. The key property is at the end of the
next definition: every point on the original grid is in the interior of exactly one of the rectangles [regions determined by
  consecutive vertical and horizontal elements] of the auxiliary grid, and indeed this is a one-to-one correspondence between
points of the original grid and rectangles of the auxiliary grid. Thus every labeling of the original grid determines a union of
rectangles of the auxiliary grid (by taking those rectangles whose interior point was labeled $1$, along with its lower and left
boundary\footnote{This ensures we do not pick up additional alternation along boundary lines. The reason why we take a
``half-open'' rectangle rather than an open one is technical and is explained in Discussion~\ref{disc:halfopen} below.}). The
hypothesis class $\cH$ we use will consist of all unions of such half-open rectangles obtained in this way when the labeling
comes from a set with at most $\alt$ alternations along any axis-parallel line.

\begin{definition}[The auxiliary grid]\label{def:auxiliarygrid}
  Let $\overline{a},\overline{b}\in\fn{m}{\RR}$.
  \begin{enumdef}
  \item To define the \emph{auxiliary sequence} of $\overline{a}$, perform the following operations:
    \begin{enumdef}
    \item Let $\overline{c}\df\langle c_i : i < r \rangle$ list $\overline{a}$ in strictly increasing 
      order without repetitions, so $r\leq m$. 
    \item The \emph{auxiliary sequence for $\overline{a}$} is the sequence $\overline{u} = \langle u_0, \dots, u_r \rangle$
      defined by:
      \begin{gather*}
        u_i
        \df
        \begin{dcases*}
          c_0 - 1, & if $i=0$,\\
          c_{r-1} + 1, & if $i=r$,\\
          \frac{c_{i-1}+c_i}{2}, & if $1\leq i\leq r-1$.
        \end{dcases*}
      \end{gather*}
      Informally, we take $u_i$ as midpoints of $c_{i-1}$ and $c_i$; when $c_{i-1}$ does not exist (because $i=0$), we
      take $u_0$ as $c_0-1$; and when $c_i$ does not exist (because $i=r$), we take $u_r$ as $c_{r-1}+1$. This ensures
      \begin{gather*}
        u_0 < c_0 < u_1 < c_1 < \cdots < u_{r-1} < c_{r-1} < u_r.
      \end{gather*}
    \end{enumdef}
  \item The \emph{auxiliary grid} determined by $(\overline{a},\overline{b})$ of elements of $\RR$, is obtained by performing
    the following operations:
    \begin{enumdef}
    \item Call $\Gamma(\overline{a}, \overline{b})$ the original grid. 
    \item Let $\overline{u} = \langle u_i : i < r_0 \rangle$ and $\overline{v} = \langle v_j : j < r_1 \rangle$ be the
      auxiliary sequences for $\overline{a}$ and $\overline{b}$, respectively. Note that since the $\overline{a}, \overline{b}$
      may have contained different numbers of repetitions, a priori after strict re-ordering $r_0$ may not equal $r_1$ although
      necessarily $r_0, r_1 \leq m+1$.
    \item Call $\Gamma(\overline{u}, \overline{v})$ the \emph{auxiliary grid}. 
    \item By a \emph{rectangle} in the auxiliary grid let us mean a region of the form
      \begin{gather*}
        [u_i,u_{i+1})\times[v_j,v_{j+1})
        =
        \{(x,y) : u_i \leq x < u_{i+1}, v_j \leq y < v_{j+1} \} \subseteq \RR^2  
      \end{gather*}
      for some $i < \lgn(\overline{u})-1$ and some $j < \lgn(\overline{v})-1$. Informally, this is the rectangle
      determined by two consecutive points of $\overline{u}$ and two consecutive points of $\overline{v}$ and that has all
      boundary points on its bottom and left edges and no boundary points on the other edges (and the only vertex it contains is
      the bottom left one).
    \end{enumdef}
    Observe that the auxiliary grid $\Gamma(\overline{u}, \overline{v})$ has the property that every point on the original grid
    $\Gamma(\overline{a}, \overline{b})$ is in the interior of exactly one of the rectangles in the auxiliary grid, and indeed
    this is a one-to-one correspondence between points of the original grid and rectangles of the auxiliary grid.
  \item Given the original grid $\Gamma = \Gamma(\overline{a}, \overline{b})$ and a labeling $\pi\colon\Gamma\to\{0,1\}$, define
    the \emph{$(\Gamma,\pi)$-auxiliary union of rectangles} to be the union of all rectangles of the auxiliary grid whose
    interior contains a point of the original grid labeled $1$.
  \end{enumdef}
\end{definition}

\begin{discussion}\label{disc:halfopen}
  Before we prove Theorem~\ref{thm:alternation}, let us discuss its main obstacle, which justifies our choice to use half-open
  rectangles in Definition~\ref{def:auxiliarygrid}.

  The idea of the proof is naturally to apply our meta-theorem, Theorem~\ref{thm:meta-thm}, in a similar fashion to what we did
  in the case of convex sets (cf.~Theorem~\ref{thm:convex}), taking $\cH$ as the class of all auxiliary unions of rectangles
  that arise from grids with at most $\alt$ alternations. Most conditions of Theorem~\ref{thm:meta-thm} will be easy to check;
  the only one that will require a reasonable amount of work is the fact that $\cH$ is a hypothesis class.

  The basic idea is to encode $\cH$ as a continuous image of a Polish space, making it a Suslin space. Based on the proof of
  Theorem~\ref{thm:meta-thm}, we might be tempted to consider the space $\cK$ of non-empty compact sets and use \emph{closed}
  rectangles instead; however, this does not quite work as we need to bound the $\SVC$-dimension of $\cH$. The issue is that
  given a labeling of a grid with at most $\alt$ alternations, if we take the auxiliary union of \emph{closed} rectangles, then
  the number of alternations can become $2\alt$ along lines in the auxiliary grid, since the boundaries of the rectangles can
  come from either side of the line. But then, since $\cH$ has to be definably realizable in itself, this spirals into an
  unbounded number of alternations.

  A naive way to fix this alternation is to consider unions of \emph{open} rectangles instead. Of course, this would make $\cH$
  not a subset of $\cK$, but we could hope that each element of $\cH$ has some canonical representative in $\cK$ so that $\cK$
  induces a topology on $\cH$. The natural canonical representative of a union $H$ of open rectangles is simply its closure
  $\overline{H}$; however, the closure operation is not injective (e.g., $((0,1)\times(0,1))\cup((1,2)\times(0,1))$ and
  $(0,2)\times(0,1)$ have the same closure).

  It turns out that using half-open rectangles solves both problems (i.e., we preserve the bound on the alternations and ensure
  injectivity of the closure operation) while allowing the proof to go through in a slightly technical but straightforward way.
  Namely, the natural map from samples and labels (as a subset of $\fn{r_0 + r_1 + r_0r_1}{\RR}$) to $\cK$ is easily seen to be
  continuous. The fact that the closure operation on half-open rectangles is injective will also be used to check measurability
  of the evaluation map $\ev$.
\end{discussion}

\begin{proofof}{Theorem~\ref{thm:alternation}}
  We will apply our meta-theorem, Theorem~\ref{thm:meta-thm}, but not directly to $\cC$, but rather a superset of it $\cC'$.

  First, we define $\cH$ as the set of all $(\Gamma,\pi)$-auxiliary unions of rectangles where $\Gamma$ ranges over all grids
  (of finite size) and $\pi$ ranges over all labelings $\pi\colon\Gamma\to\{0,1\}$ that have at most $\alt$ alternations (regardless
  of whether such $\pi$ comes from some element of $\cC$ or not). Clearly $\cH$ is a family with $\alt$-bounded alternation.

  We then take $\cC'\df\cC\cup\cH$, which clearly is also a family with $\alt$-bounded alternation. We now verify that $(\cH,\cC')$
  satisfy the conditions of Theorem~\ref{thm:meta-thm}, that is, we need to verify that:
  \begin{enumerate*}[label={(\arabic*)}]
  \item\label{thm:alternation:subset} $\cH\subseteq\cC'$,
  \item\label{thm:alternation:measurable} all elements of $\cC'$ (and hence of $\cH$) are measurable,
  \item\label{thm:alternation:SVC} $\cH$ has finite slicewise $\VC$-dimension,
  \item\label{thm:alternation:defreal} $\cC'$ is $\cH$-definably realizable on grids, and
  \item\label{thm:alternation:class} $\cH$ is a hypothesis class.
  \end{enumerate*}

  \begin{description}[wide, itemsep={3ex}]
  \item[Item~\ref{thm:alternation:subset}] This follows from construction as $\cC'=\cC\cup\cH$.
  \item[Item~\ref{thm:alternation:measurable}] All elements of $\cC$ are measurable by assumption and all elements of $\cH$ are
    compact sets, hence measurable; thus, all elements of $\cC'=\cC\cup\cH$ are measurable.
  \item[Item~\ref{thm:alternation:SVC}] By construction, $\cH$ is a family with $\alt$-bounded alternation, hence has slicewise
    $\VC$-dimension at most $\alt+1$.
  \item[Item~\ref{thm:alternation:defreal}] For each $m < \omega$, let $\varphi_m(x,y;\overline{z}_m,\overline{w}_m)$ be an
    $\cL$-formula that defines the $(\Gamma,\pi)$-auxiliary union of rectangles, where $\Gamma$ is the grid defined by the
    variables $\overline{z}_m$ and $\pi$ is the labeling of $\Gamma$ defined by the variables of $\overline{w}_m$. It is clear
    from Definition~\ref{def:auxiliarygrid} that such an $\cL$-formula exists.

    We now argue that $\cC'$ is $\cH$-definably realizable. All elements of $\cC'$ have at most $\alt$ alternations, so all
    labelings they induce on grids also have at most $\alt$ alternations, so the corresponding set defined by $\varphi_m$ is in
    $\cH$ and clearly coincides on the grid with the labeling that generated it, so all properties of Definition~\ref{def:Hdr}
    hold.
  \item[Item~\ref{thm:alternation:class}] We now justify that $\cH$ is a hypothesis class (this is motivated in
    Discussion~\ref{disc:halfopen} above). Similarly to the proof of Theorem~\ref{thm:convex}, by Remark~\ref{rmk:adding-F}, it
    suffices to show that $\cH'\df\cH\setminus\{\varnothing\}$ is a hypothesis class and we will make use of the Polish space
    $\cK$ of non-empty compact sets equipped with (topology induced by) Hausdorff metric as well as the space $\cH''$ of all
    closures of elements of $\cH'$.

    First, we note that the closure operation $\cH'\ni H\mapsto\overline{H}\in\cH''$ is a bijection. Indeed, surjection is by
    definition of $\cH''$ and injection follows since the inverse $i\colon\cH''\to\cH$ is given by
    \begin{multline*}
      i(A)
      \df
      \{(x,y)\in\RR^2 : \exists (x_n,y_n)_{n < \omega} \in \fn{\omega}{A}, \lim_{n\to\infty} (x_n,y_n) = (x,y)\land
      \\
      \land
      \forall n < \omega, (x_n > x_{n+1}\land y_n > y_{n+1})
      \}.
    \end{multline*}

    Since $\cH''\subseteq\cK$, we can equip it with the induced topology by the Hausdorff distance and we equip $\cH'$ with the
    topology induced via the bijection above (i.e., we make $\cH'$ homeomorphic to $\cH''$).

    Our goal is to show that $\cH''$ (hence also $\cH'$) is a Suslin space.
    
    For each $r < \omega$, let $\fn{r}{\RR_<}$ be the set of $r$-tuples $\overline{a}\in\fn{r}{\RR}$ that are strictly
    increasing (i.e., $a_0 < a_1 < \cdots < a_{r-1}$) and for $r_0,r_1 < \omega$ let
    \begin{multline*}
      B_{r_0,r_1}
      \df
      \{\ell\in\fn{r_0\times r_1}{\{0,1\}} : \exists i < r_0, \exists j < r_1, \ell(i,j)=1
      \\
      \land\forall i < r_0, \forall j_0 < j_1 < \cdots < j_{\alt+1} < r_1,
      \neg(\forall t < \alt+1, \ell(i,j_t)\neq \ell(i,j_{t+1}))
      \\
      \land
      \forall i_0 < i_1 < \cdots < i_{\alt+1} < r_0, \forall j < r_1,
      \neg(\forall t < \alt+1, \ell(j_t,i)\neq \ell(j_{t+1},i))
      \}
    \end{multline*}
    be the set of all non-zero $\ell\colon r_0\times r_1\to\{0,1\}$ that have at most $\alt$ alternations in every row and every
    column.

    Then there is a natural encoding of elements of $\cH''$ by $\bigcup_{r_0,r_1 < \omega}
    (\fn{r_0}{\RR_<}\times\fn{r_1}{\RR_<}\times B_{r_0,r_1})$, that is, there is a function $c\colon\bigcup_{r_0,r_1 < \omega}
    (\fn{r_0}{\RR_<}\times\fn{r_1}{\RR_<}\times B_{r_0,r_1})\to\cK$ that takes in $(\overline{a},\overline{b},\ell)$, constructs
    the grid $\Gamma(\overline{a},\overline{b})$, uses $\ell$ to create a labeling $\pi$ of $\Gamma$ (via
    $\pi(a_i,b_j)\df\ell(i,j)$), then outputs the $(\Gamma,\pi)$-auxiliary union of rectangles. Note that $\cH''$ is exactly the
    image of $c$.

    For each $r_0,r_1 < \omega$, we equip each component $\fn{r_0}{\RR_<}\times\fn{r_1}{\RR_<}\times B_{r_0,r_1}$ with the
    Euclidean distance (as a subset of $\fn{r_0+r_1+r_0r_1}{\RR}$). We then equip the union $\bigcup_{r_0,r_1 < \omega}
    (\fn{r_0}{\RR_<}\times\fn{r_1}{\RR_<}\times B_{r_0,r_1})$ with the distance that uses the Euclidean distance in each factor
    and sets the distance between points in different factors to be $1$. It is straightforward to check that this is a metric
    that induces a topology that is homeomorphic to the usual coproduct topology in $\coprod_{r_0,r_1 < \omega} \coprod_{\ell\in
      B_{r_0,r_1}} \fn{r_0 + r_1}{\RR}$, hence $\bigcup_{r_0,r_1 < \omega}(\fn{r_0}{\RR_<}\times\fn{r_1}{\RR_<}\times
    B_{r_0,r_1})$ is a Polish space.

    We now claim the encoding function $c$ is continuous. Indeed, we claim that $c$ is locally Lipschitz: namely, given points
    $(\overline{a},\overline{b},\ell), (\overline{a}',\overline{b}',\ell')
    \in\bigcup_{r_0,r_1<\omega}\fn{r_0}{\RR_<}\times\fn{r_1}{\RR_<}\times B_{r_0,r_1}$ that are at distance $\epsilon < 1$, then
    it must be the case that the tuples come from the same component corresponding to the same $(r_0,r_1)$ and must have the
    same labeling $\ell=\ell'$ (this is because $\epsilon < 1$ and different factors are at distance $1$ and different labelings
    of the same factor are at distance at distance $1$).

    It also clear that since these points are at distance $\epsilon$, then the points of the auxiliary grid of
    $(\overline{a},\overline{b})$ must be at distance at most $\epsilon$ from the corresponding ones of
    $(\overline{a}',\overline{b}')$, which in turn implies that the outputs of $c$ on these two points are at Hausdorff distance
    at most $\epsilon$ from each other.

    Since $\cH''$ is exactly the image of $c$, it is the continuous image of a Polish space, hence it is a Suslin space. Thus,
    as $\cH'$ was made homeomorphic to $\cH''$, it is also a Suslin space, so for every standard Borel space $\Upsilon$ and
    every measurable set $Y\subseteq\cH'\times\Upsilon$, the projection of $Y$ onto $\Upsilon$, i.e., the set
    $\{\upsilon\in\Upsilon : \exists H\in\cH', (H,\upsilon)\in Y\}$, is a Suslin set, which is universally measurable.

    \smallskip

    We now need to argue that the evaluation map $\ev\colon\cH'\times\RR^2\to\{0,1\}$ is measurable. This is slightly more
    complicated than our argument of Theorem~\ref{thm:convex} as the topology on $\cH'$ is induced by the one in $\cH''$ and the
    evaluation map over $\cH''$ is not the same as the one on $\cH'$ because of the boundary points. Nevertheless, note that if
    $H\in\cH'$ and $\overline{x}\in\RR^2$, then since $H$ is a finite union of rectangles of the form $[a,b)\times[c,d)$, we get
    \begin{gather*}
      \overline{x}\in H
      \iff
      \exists n < \omega, \forall m > n, \overline{x} + \left(\frac{1}{m},\frac{1}{m}\right)\in\overline{H},
    \end{gather*}
    from which we conclude that
    \begin{gather*}
      \ev(H,\overline{x})
      =
      \One[\overline{x}\in H]
      =
      \One\left[
        \overline{x}
        \in
        \bigcup_{n < \omega}\bigcap_{m > n} \left(\overline{H} - \left(\frac{1}{m},\frac{1}{m}\right)\right)
        \right]
      =
      \liminf_{n\to\infty}
      \One\left[d\left(\overline{x},\overline{H} - \left(\frac{1}{n},\frac{1}{n}\right)\right) = 0\right].
    \end{gather*}

    Now, by Lemma~\ref{lem:Lipschitz}, we know that the point-to-set distance $\RR^2\times\cK\ni(\overline{x},K)\mapsto
    d(\overline{x},K)\in\cK$ is Lipschitz-continuous and it is straightforward to check that translation
    $\RR^2\times\cK\ni(\overline{x},K)\mapsto(K + \overline{x})\in\cK$ is also Lipschitz-continuous; this means that for every
    $n < \omega$, the function $(\overline{x},H)\mapsto\One[d(\overline{x},\overline{H} - (1/n,1/n))=0]$ is measurable, hence
    the evaluation map is also measurable as the limit inferior of measurable functions.
  \end{description}

  Thus, Theorem~\ref{thm:meta-thm} applies.
\end{proofof}

\begin{conclusion}
  Just as in Theorem~\ref{thm:convex}, we could have even made a stronger statement that for every choice of $\mu_0,\mu_1$ there
  is a single $m\times m$ grid $\Gamma$ in $\RR^2$ so that \emph{every} set $C\in\cC$ is $\epsilon$-close (in the sense of
  $\mu_0\otimes\mu_1$) to some definable set given by an instance of $\varphi_m$ with parameters from $\Gamma$. (However, recall
  from Discussion~\ref{disc:gridlabeling}, if we choose this stronger version with a single grid for all elements $C\in\cC$,
  then the labeling of such a grid might not be precisely the same as what $C$ does on it.)

  Furthermore, the definable sets that approximate the elements of $\cC$ are all uniformly bounded unions of half-open
  rectangles. More precisely, they are obtained as follows: our formula $\varphi_m$ takes in a representative grid, forms the
  auxiliary grid (here we use that we are in an ordered field and not just an ordered structure) and takes the union of the half-open
  rectangles with endpoints given by the proof.
\end{conclusion}

\begin{discussion}
  This conclusion points out that Theorem~\ref{thm:alternation} can be seen as a kind of approximate cell decomposition: it
  tells us some key aspects of the good behavior of the actual cell decomposition theorem for o-minimal structures extend
  (approximately but definably) to undefinable families which are a priori well beyond the reach of the real field, but satisfy
  bounded alternation which retains a spirit of o-minimality.
\end{discussion}

\section{Approximation on finite grids}

In this section we specialize the previous results to grids of the form $A \times B$ where $A$, $B$ are finite (but potentially
very large). In this case, per our global conventions, we use the discrete $\sigma$-algebra and so all sets are measurable (and
the space is standard Borel) and our statements become rather broad. $\cC$ can be any family of sets in the plane of bounded
alternation. In fact, our ambient model $M$ containing the grid $A\times B$ can be any ordered field, and is no longer
constrained to be $\RR$.

\begin{theorem}\label{thm:finite-grids}
  For any $\epsilon > 0$ and $\alt < \omega$, there is $m = m(\epsilon,\alt)$ and an $\cL$-formula $\varphi(x,y; \overline{z},
  \overline{w})$ with $\lgn(\overline{z}\conc\overline{w}) = 2m + m^2$ so that for any family $\cC$ with $\alt$-bounded
  alternation, any $C\in\cC$, any $n < \omega$, any $n\times n$ grid $\Gamma = \Gamma(\overline{a}, \overline{b}) = \{(a_i,b_j)
  : i,j < n\}$ in $\RR^2$, and any choice of Borel probability measures $\mu_0$ on $\{ a_i : i < n \}$ and $\mu_1$ on $\{ b_j :
  j < n \}$, there are $2m$ parameters $\overline{c}$ from $\Gamma$ and $\overline{d}\in\fn{m^2}{\{0,1\}}$ such that:
  \begin{gather*}
    (\mu_0\otimes\mu_1) \bigl(\varphi(M^2;\overline{c}\conc\overline{d}) \rstr_\Gamma\symdiff C \rstr_\Gamma\bigr)\leq\epsilon.
  \end{gather*}
  Moreover, such parameters are easy to find in the sense of the uniform convergence statement in Appendix~\ref{sec:UC}.
\end{theorem}

Informally, Theorem~\ref{thm:finite-grids} says that there is $m = m(\epsilon,\alt)$ and an $\cL$-formula with $2m$ parameter
variables plus $m^2$ variables taking in the values $0$ or $1$ (where $m$ and $\varphi$ depend on $\epsilon$ and $\cC$ but are
independent of grid, measures, or the specific element of $\cC$) with the following property. For any finite grid $\Gamma
\subseteq M^2$, of any size, for any choice of probability measures $\mu_0$ on $\{ a_i : i < n \}$ and $\mu_1$ on $\{ b_j : j <
n \}$, we can subsample $2m$ parameters from $\Gamma$ so that the number of points on $\Gamma$ belonging to our instance of
$\varphi$ is $\epsilon$-close to the number of points on $\Gamma$ belonging to $C$, in the sense of the product measure $\mu_0
\otimes \mu_1$. (The choice of parameters depends, of course, on the measures and on the choice of $C$.) Moreover, ``many such
subsamples will work'' where the formal dependency is given in Appendix~\ref{sec:UC}.

\begin{proof}
  The proof is the same as that of Theorem~\ref{thm:alternation}, but we take as our base spaces $\{a_i : i < n\}$ and $\{b_j :
  j < n\}$, which are standard Borel spaces as they are finite spaces.
\end{proof}

\printbibliography

\appendix

\section{High-arity uniform convergence}
\label{sec:UC}

In this appendix we formally state the high-arity uniform convergence result proved in~\cite{CM24} as we will use it here. We
start by formally stating the theorem that we use in this paper:
\begin{theorem}[Simplified two-way high-arity uniform convergence with $0$-$1$-loss and binary labeling, for
    $k=2$]\label{thm:twoUC}
  There exist absolute constants
  \begin{align*}
    0 & < K < 6873, &
    0 & < K' < 8109.
  \end{align*}
  such that the following holds: for every $d < \omega$ and every $\epsilon,\delta > 0$, there exists
  \begin{gather*}
    m^{\UC} = m^{\UC}(\epsilon,\delta,d) \leq \frac{Kd}{\delta^2\epsilon^2}\cdot\ln\frac{K'd}{\delta^2\epsilon^2}
  \end{gather*}
  such that for all standard Borel spaces $\Omega_0=(X_0,\cB_0)$ and $\Omega_1=(X_1,\cB_1)$, for all probability measures
  $\mu_0$ and $\mu_1$ on $\Omega_0$ and $\Omega_1$, respectively, for every hypothesis class $\cH\subseteq\cP(X_0\times X_1)$
  with $\SVC(\cH)\leq d$, for every measurable $F\subseteq X_0\times X_1$ and every positive integer $m\geq m^{\UC}$, we have
  \begin{gather*}
    \PP_{\substack{\rn{\overline{a}}\sim\mu_0^m\\\rn{\overline{b}}\sim\mu_1^m}}\Bigl[
      \forall H_0,H_1\in\cH\cup\{F\},
      \bigl\lvert
      L_{\mu_0,\mu_1}(H_0,H_1)
      -
      L_{\rn{\overline{a}},\rn{\overline{b}}}(H_0,H_1)
      \bigr\rvert
      \leq
      \epsilon
      \Bigr]
    \geq
    1 - \delta,
  \end{gather*}
  where
  \begin{align*}
    L_{\mu_0,\mu_1}(H_0,H_1)
    & \df
    (\mu_0\otimes\mu_1)(H_0\symdiff H_1),
    \\
    L_{\overline{a},\overline{b}}(H_0,H_1)
    & \df
    \frac{\lvert\{(i,j)\in m\times m : (a_i,b_j)\in H_0\symdiff H_1\}\rvert}{m^2}.
  \end{align*}
\end{theorem}

We directly quote the uniform convergence property that was defined and proved in~\cite{CM24} before giving a simplified
form that suffices for this paper, and we point out that the simplified form we will use can also be retrieved from
Livni--Mansour~\cite{LM19a,LM19b}. 
\begin{center}
  \includegraphics[width=0.8\linewidth]{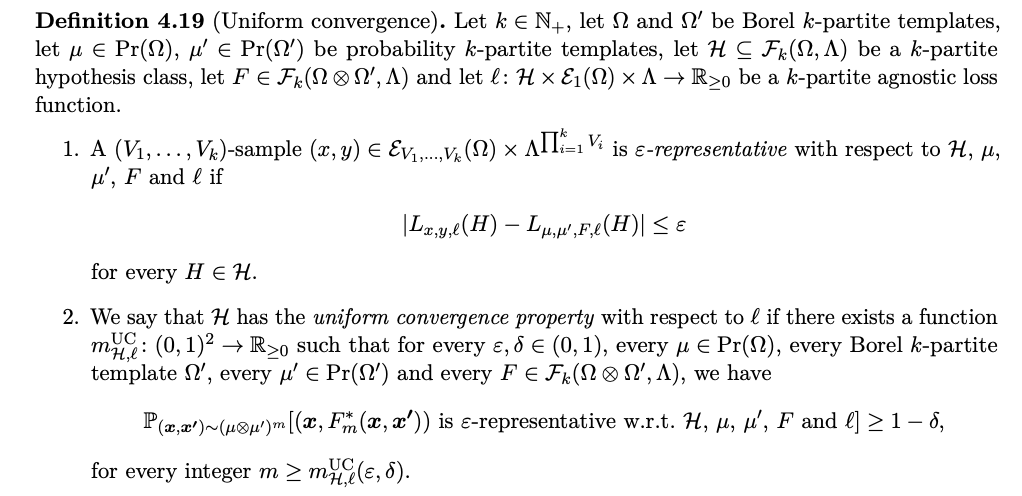}

  \includegraphics[width=0.8\linewidth]{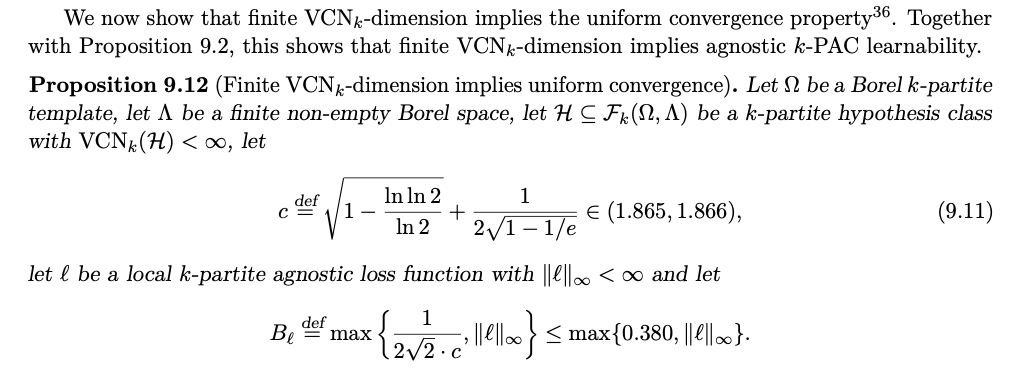}

  \includegraphics[width=0.8\linewidth]{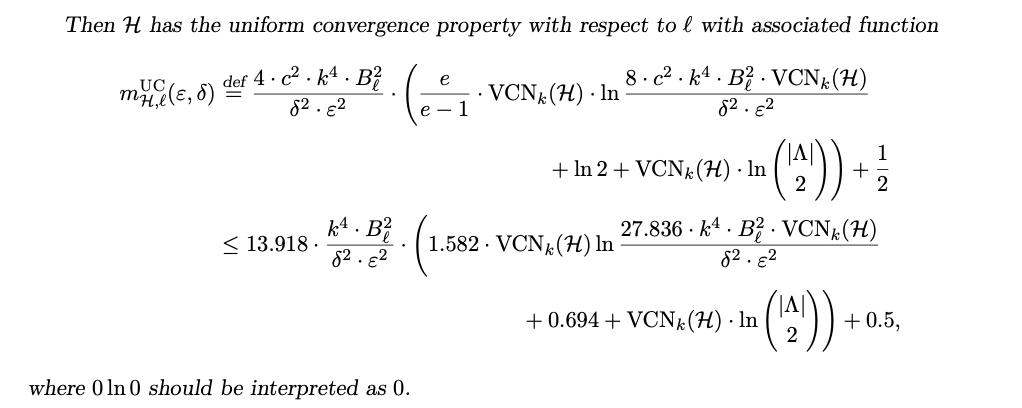}

  \includegraphics[width=0.8\linewidth]{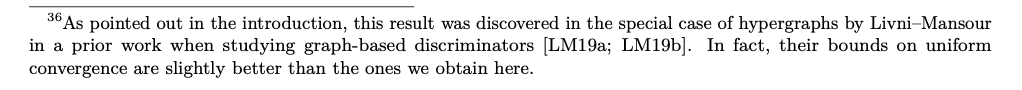}
\end{center}

We won't give all of the constituent definitions in their full generality, but we give a short description of the additional
features present in the above that are not used in this work; we refer readers interested in these features to the expository
Section~2 of~\cite[\S 2]{CM24}.
\begin{enumerate}
\item Obviously, the above covers general arity $k < \omega$ and not just the case $k=2$ used here\footnote{Let us also point
out that in set theory language, one writes $f\colon k\to m$ and starts counting at zero; but~\cite{CM24} was written in a more
combinatorial language, in which one writes $f\colon [k]\to [m]$ and starts counting at one.}.
\item It also allows for arbitrary standard Borel spaces (although, all such spaces are Borel-isomorphic to a Borel subset of
  $\RR$), i.e., instead of having $\RR^k$, we consider $\Omega_1\times\cdots\times\Omega_k$.
\item Hypotheses are more general than just subsets $H\subseteq\Omega_1\times\cdots\times\Omega_k$, they are allowed to be
  functions $H\colon\Omega_1\times\cdots\times\Omega_k\to\Lambda$, where $\Lambda$ is a finite set of labels (the case of sets
  is retrieved by taking $\Lambda=\{0,1\}$ with the set encoded by the preimage of $1$). In turn, the corresponding combinatorial
  dimension $\VCN_k$, is a version of the slicewise $\VC$-dimension that is based on the Natarajan dimension instead (hence the
  ``N'' added to ``$\VC$'').
\item It uses an arbitrary loss function $\ell$ that takes in a hypothesis $H$, a tuple $\overline{a}$ and a label $y\in\Lambda$
  and outputs a ``penalty'' $\ell(H,\overline{a},y)\in\RR_{\geq 0}$ that $H$ incurs at the point $\overline{a}$ when compared to
  the label $y$. The total loss of $H$ versus some $F$ versus some tuple of measures $\mu$ is computed as the expected value of
  $\ell(H,\rn{\overline{a}},F(\rn{\overline{a}}))$ when $\rn{\overline{a}}$ is picked at random according to the product measure
  of $\mu$. Our setup is retrieved by taking the $0$-$1$-loss that simply outputs $1$ if $H(\overline{a})\neq y$ and outputs $0$
  otherwise.

  In the theorem, the agnostic loss is further required to be bounded (in the usual sense) and local, which means that it
  factors as the sum of a value $\ell'(\overline{a},H(\overline{a}),y)$ that depends only on the point $\overline{a}$, the label
  $y$ and the labels that $H$ assigns to $\overline{a}$ with a value $r(H)$ that only depends on $H$. (Clearly, the $0$-$1$-loss
  is local as it only depends on $y$ and $H(\overline{a})$.)
\item The uniform convergence is actually computed against an ``agnostic measure-function pair'' $(\nu,F)$ rather than a usual
  measure-function pair $(\mu,F)$. Without going into technical details, this means that there can be hidden randomness
  variables ($\rn{x'}$) that introduce noise into the system: the label of a tuple $\rn{x}$ is decided based not only on the
  elements of $\rn{x}$ but also of hidden random variables $\rn{x'}$ that are sampled for every finite subset of $\omega$
  according to a local rule $F$ (this precisely captures local separately exchangeable distributions via the Aldous--Hoover
  Theorem~\cite{Ald81,Ald85,Hoo79}).
\item The hypotheses are also allowed to be Aldous--Hoover representations, which means that the underlying space
  $\cE_1(\Omega)$ considerably more technical than simply the product $\Omega_1\times\cdots\times\Omega_k$.
\end{enumerate}

We now state the simplified version of the above that we need to prove Theorem~\ref{thm:twoUC}:
\begin{theorem}[High-arity uniform convergence, $k=2$]\label{thm:oneUC}
  There exist absolute constants
  \begin{align*}
    0 & < K < 378, &
    0 & < K' < 446.
  \end{align*}
  such that the following holds: for every $d < \omega$ and every $\epsilon,\delta > 0$, there exists
  \begin{gather*}
    m^{\UC} = m^{\UC}(\epsilon,\delta,d) \leq \frac{Kd}{\delta^2\epsilon^2}\cdot\ln\frac{K'd}{\delta^2\epsilon^2}
  \end{gather*}
  such that for all standard Borel spaces $\Omega_0=(X_0,\cB_0)$ and $\Omega_1=(X_1,\cB_1)$, for all probability measures
  $\mu_0$ and $\mu_1$ on $\Omega_0$ and $\Omega_1$, respectively, for every hypothesis class $\cH\subseteq\cP(X_0\times X_1)$
  with $\SVC(\cH)\leq d$, for every measurable $F\subseteq X_0\times X_1$ and every positive integer $m\geq m^{\UC}$, we have
  \begin{gather*}
    \PP_{\substack{\rn{\overline{a}}\sim\mu_0^m\\\rn{\overline{b}}\sim\mu_1^m}}\Bigl[
      \forall H\in\cH,
      \bigl\lvert
      L_{\mu_0,\mu_1}(F,H)
      -
      L_{\rn{\overline{a}},\rn{\overline{b}}}(F,H)
      \bigr\rvert
      \leq
      \epsilon
      \Bigr]
    \geq
    1 - \delta,
  \end{gather*}
  where
  \begin{align*}
    L_{\mu_0,\mu_1}(F,H)
    & \df
    (\mu_0\otimes\mu_1)(F\symdiff H),
    \\
    L_{\overline{a},\overline{b}}(F,H)
    & \df
    \frac{\lvert\{(i,j)\in m\times m : (a_i,b_j)\in F\symdiff H\}\rvert}{m^2}.
  \end{align*}
\end{theorem}

Note that Theorem~\ref{thm:oneUC} is a priori weaker than Theorem~\ref{thm:twoUC}: it does not allow comparison of two arbitrary
elements of $\cH$, rather it only allows comparison of an arbitrary element of $\cH$ with the measurable $F$ fixed in advance
(more simply, it would be as if in Theorem~\ref{thm:twoUC} we required $H_0$ to be $F$). To upgrade Theorem~\ref{thm:oneUC}
to~\ref{thm:twoUC}, we recall some concepts and basic facts of learning theory:
\begin{definition}\label{def:VC}
  A family $\cH\subseteq\cP(X)$ of subsets of a set $X$ \emph{shatters} a set $U$ if for
  \begin{gather*}
    \cH\rest_U \df \{H\cap U : H\in\cH\}
  \end{gather*}
  we have $\cH\rest_U = \cP(U)$, that is, all possible patterns on $U$ are realized by some element of $H$.

  The \emph{Vapnik--\Cervonenkis ($\VC$) dimension} of $\cH$ is defined as
  \begin{gather*}
    \VC(\cH) \df \sup\{\lvert U\rvert : \text{$U$ finite and $\cH$ shatters $U$}\}.
  \end{gather*}
\end{definition}

\begin{lemma}[Sauer~\cite{Sau72}--Shelah~\cite{She72}--Perles~\cite{Per72}]\label{lem:SSP}
  Let $\cH\subseteq\cP(X)$ be a family with $\VC(\cH) < \infty$ and let $U\subseteq X$ be a finite set, then
  \begin{gather*}
    \lvert\cH\rest_U\rvert \leq \sum_{i\leq\VC(\cH)} \binom{\lvert U\rvert}{i}.
  \end{gather*}
\end{lemma}

To upgrade to two-way, we will use a simple trick involving the symmetric difference family $\cH\otriangle\cH$ below:
\begin{lemma}\label{lem:symdiff}
  Let $\cH\subseteq\cP(X)$ and let
  \begin{equation*}
    \cH\otriangle\cH \df \{H_1\symdiff H_2 : H_1,H_2\in\cH\}.
  \end{equation*}
  Then we have
  \begin{gather*}
    \VC(\cH\otriangle\cH)
    \leq
    C\cdot\VC(\cH),
  \end{gather*}
  where
  \begin{gather*}
    C
    \df
    \max\left\{c\geq 1 : h_2\left(\frac{1}{c}\right) \leq \frac{1}{2}\right\}
    \leq
    9.09,
    \\
    h_2(p)
    \df
    p\log_2\frac{1}{p} + (1-p)\log_2\frac{1}{1-p}.
  \end{gather*}
\end{lemma}

\begin{proof}
  Let $d\df\VC(\cH)$. The result is trivial if $d = \infty$, so we suppose $d < \infty$. Let $U\subseteq X$ be a finite set
  shattered by $\cH\otriangle\cH$ and let $m\df\lvert U\rvert$. Our goal is to show that $m\leq C\cdot d$ and since $C\geq 9$,
  we may assume $d\leq m/2$ otherwise there is nothing to prove. Note that
  \begin{gather*}
    2^m
    =
    \bigl\lvert(\cH\otriangle\cH)\rest_U\bigr\rvert
    \leq
    \bigl\lvert\cH\rest_U\bigr\rvert^2
    \leq
    \left(\sum_{i\leq d}\binom{m}{i}\right)^2
    \leq
    \bigl(2^{h_2(d/m)\cdot m}\bigr)^2,
  \end{gather*}
  where the second inequality is by Lemma~\ref{lem:SSP} and the third inequality is the standard bound on the size of layers of
  the hypercube as $d\leq m/2$ (see e.g.~\cite[Lemma~4.7.2]{Ash65}).

  Thus, we conclude that
  \begin{gather*}
    h_2\left(\frac{d}{m}\right) \leq \frac{1}{2},
  \end{gather*}
  so by the definition of $C$, we get $m/d\leq C$, hence $m\leq C\cdot d$ as desired.
\end{proof}

We can now prove Theorem~\ref{thm:twoUC} from Theorem~\ref{thm:oneUC}:
\begin{proofof}{Theorem~\ref{thm:twoUC}}
  Since we will reduce this result to Theorem~\ref{thm:oneUC}, all notation corresponding to it will use tildes, while the
  notation of the current theorem will not.
  
  Let $\widetilde{K}$ and $\widetilde{K'}$ be the absolute constants of Theorem~\ref{thm:oneUC} and let $C$ be as in
  Lemma~\ref{lem:symdiff} and set
  \begin{align*}
    K & \df 2C\widetilde{K} < 6873,
    &
    K' & \df 2C\widetilde{K'} < 8109.
  \end{align*}
  The result is trivial if $d = 0$ (as $\cH$ must then consist of at most one element), so we suppose $d > 0$. We now take
  \begin{gather*}
    m^{\UC}(\epsilon,\delta,d)
    \df
    \widetilde{m}^{\UC}(\epsilon,\delta,\floor{C(d+1)})
    \leq
    \frac{\widetilde{K}C(d+1)}{\delta^2\epsilon^2}\cdot\ln\frac{\widetilde{K'}C(d+1)}{\delta\epsilon}
    \leq
    \frac{Kd}{\delta^2\epsilon^2}\cdot\ln\frac{K'd}{\delta^2\epsilon^2},
  \end{gather*}
  where the last inequality follows since $d\geq 1$ (so $C(d+1)\leq 2Cd$).

  Consider a hypothesis class $\cH$ with $\SVC(\cH)\leq d$ and let $\cH'\df\cH\cup\{F\}$ and note that
  $\SVC(\cH')\leq\SVC(\cH)+1\leq d+1$. By Lemma~\ref{lem:symdiff} applied to each slice of $\cH'$, we have
  $\SVC(\cH'\otriangle\cH')\leq C(d+1)$. For a positive integer $m\geq m^{\UC}$, using the guarantee of Theorem~\ref{thm:oneUC}
  for $\cH'\otriangle\cH'$ against the empty set $F=\varnothing$, we get
  \begin{gather}\label{eq:twoUC}
    \PP_{\substack{\rn{\overline{a}}\sim\mu_0^m\\\rn{\overline{b}}\sim\mu_1^m}}\Bigl[
      \forall H\in\cH'\otriangle\cH',
      \bigl\lvert
      L_{\mu_0,\mu_1}(\varnothing, H)
      -
      L_{\rn{\overline{a}},\rn{\overline{b}}}(\varnothing, H)
      \bigr\rvert
      \leq
      \epsilon
      \Bigr]
    \geq
    1 - \delta.
  \end{gather}

  But note that for every $H_0,H_1\in\cH\cup\{F\}=\cH'$, we have $H_0\symdiff H_1\in\cH'\otriangle\cH'$ and
  \begin{align*}
    L_{\mu_0,\mu_1}(\varnothing, H_0\symdiff H_1)
    & =
    L_{\mu_0,\mu_1}(H_0,H_1),
    \\
    L_{\rn{\overline{a}},\rn{\overline{b}}}(\varnothing, H_0\symdiff H_1)
    & =
    L_{\rn{\overline{a}},\rn{\overline{b}}}(H_0,H_1),
  \end{align*}
  so~\eqref{eq:twoUC} directly translates to the conclusion of the theorem.
\end{proofof}

\section{Hausdorff metric and convex sets}
\label{sec:Hausdorff}

For the purposes of self-contanment, we now provide short proofs of the simple lemmas on Hausdorf distance.
\begin{lemma}\label{lem:Lipschitz}
  For $\overline{x},\overline{y}\in\RR^2$ and sets $A,B\subseteq\RR^2$, we have
  \begin{gather*}
    \lvert d(\overline{x},A) - d(\overline{y}, B)\rvert
    \leq
    d(\overline{x},\overline{y}) + d_H(A,B),
  \end{gather*}
  where $d$ is the Euclidean distance and $d_H$ is the Hausdorff distance (induced by Euclidean distance).
\end{lemma}

\begin{proof}
  Triangle inequality implies that $d(\overline{x},A)\leq d(\overline{x},\overline{y}) + d(\overline{y},A)$, so it suffices to
  show that
  \begin{gather*}
    d(\overline{y},A) \leq d(\overline{y},B) + d_H(A,B).
  \end{gather*}

  For $\epsilon > 0$, let $\overline{z}\in B$ be such that $d(\overline{y},\overline{z})\leq d(\overline{y},B)+\epsilon$. By the
  definition of Hausdorff distance $d_H(A,B)\df\max\{\sup_{\overline{a}\in A} d(\overline{a},B), \sup_{\overline{b}\in B}
  d(\overline{b},A)\}$, we must have $d(\overline{z},A)\leq d_H(A,B)$, so
  \begin{gather*}
    d(\overline{y},A)
    \leq
    d(\overline{y},\overline{z}) + d(\overline{z},A)
    \leq
    d(\overline{y},B) + \epsilon + d_H(A,B)
  \end{gather*}
  and letting $\epsilon\to 0$ concludes the proof.
\end{proof}
  
\begin{lemma}\label{lem:convexHausdorff}
  The space $\cH'$ of non-empty compact convex sets is a closed subspace of the space of non-empty compact sets with respect to
  the Hausdorff distance $d_H$.
\end{lemma}

\begin{proof}
  Suppose $(H_n)_{n < \omega}$ is a sequence of non-empty compact convex sets conveging to some non-empty compact set $C$ and
  suppose for a contradiction that $C$ is not convex. Then there exist $\overline{x},\overline{y}\in C$ and $t\in[0,1]$ such
  that $\overline{z}\df t\overline{x} + (1-t)\overline{y}\notin C$. Since $C$ is compact (hence closed), the point
  $\overline{z}$ must be at a positive distance $\epsilon\df d(\overline{z},C) > 0$ of $C$. Let $n < \omega$ be large enough so
  that $d_H(H_n,C)\leq\epsilon/3$.

  By the definition of Hausdorff distance (and since $H_n$ is closed), there must exist points $\overline{x}_n,\overline{y}_n\in
  H_n$ with $d(\overline{x},\overline{x}_n),d(\overline{y},\overline{y}_n)\leq\epsilon/3$.

  Let $\overline{z}_n\df t\cdot\overline{x}_n + (1-t)\cdot\overline{y}_n$ and note that
  \begin{gather*}
    d(\overline{z},\overline{z}_n)
    =
    t\cdot d(\overline{x},\overline{x}_n) + (1-t)\cdot d(\overline{y},\overline{y}_n)
    \leq
    t\cdot\frac{\epsilon}{3} + (1-t)\cdot\frac{\epsilon}{3}
    =
    \frac{\epsilon}{3}.
  \end{gather*}
  On the other hand, since $H_n$ is convex, we must have $\overline{z}_n\in H_n$, hence $d(\overline{z}_n,H_n)=0$. But then
  Lemma~\ref{lem:Lipschitz} gives
  \begin{gather*}
    \epsilon
    =
    d(\overline{z},C)
    \leq
    d(\overline{z},\overline{z}_n) + d(\overline{z}_n,H_n) + d_H(H_n,C)
    \leq
    \frac{\epsilon}{3} + 0 + \frac{\epsilon}{3}
    =
    \frac{2\epsilon}{3},
  \end{gather*}
  a contradiction.
\end{proof}

\end{document}